\documentclass[12pt]{amsart}
\usepackage{amstext,amsfonts,amssymb,amscd,amsbsy,amsmath,verbatim, mathrsfs, fullpage}
\usepackage[alphabetic,abbrev,lite]{amsrefs} % for bibliography 
\usepackage{ifthen,tikz}
\usepackage{color}
\usepackage{amsthm}
\usepackage{latexsym}
\usepackage[all]{xy}
\usepackage{enumerate}
\usepackage{mathtools}

\newtheorem{lemma}{Lemma}[section]

\newtheorem{prop}[lemma]{Proposition}
\newtheorem{cor}[lemma]{Corollary}

\newtheorem{claim*}{Claim}
\newtheorem{thm}[lemma]{Theorem}

\newtheorem{question}[lemma]{Question}

\theoremstyle{definition}

\newtheorem{example}[lemma]{Example}
\newtheorem{construction}[lemma]{Construction}

\theoremstyle{remark}
\newtheorem{remark}[lemma]{Remark}
\newtheorem{remarks}[lemma]{Remarks}

\numberwithin{equation}{section}

\newcommand{\cC}{\mathcal{C}}

\newcommand{\m}{\mathfrak m}
\newcommand{\PP}{\mathbb P}

\newcommand{\HH}{H}

\newcommand{\ZZ}{\mathbb Z}
\newcommand{\QQ}{\mathbb Q}

\newcommand{\im}{\operatorname{im}}

\newcommand{\id}{\operatorname{id}}

\newcommand{\Tor}{\operatorname{Tor}}

\newcommand{\Hom}{\operatorname{Hom}} %done

\newcommand{\rank}{\operatorname{rank}}

\newcommand{\defi}[1]{\textsf{#1}} % for defined terms

\newcommand{\beq}{\begin{displaymath}}
\newcommand{\eeq}{\end{displaymath}}

\def\reg{\operatorname{reg}}

\def\nc{\newcommand}
\def\on{\operatorname}
\nc{\Q}{\mathbb{Q}}
\nc{\RR}{\mathbf{R}}
\nc{\LL}{\mathbf{L}}
\nc{\xra}{\xrightarrow}
\nc{\xla}{\xleftarrow}
\def\a{\alpha}

\def\DM{\operatorname{DM}}

\nc{\into}{\hookrightarrow}
\nc{\onto}{\twoheadrightarrow}
\nc{\OO}{\mathcal{O}}
\nc{\Z}{\mathbb{Z}}
\nc{\cA}{\mathcal{A}}
\nc{\w}{\widehat}
\nc{\End}{\on{End}}
\nc{\res}{\frac{1}{x_0x_1}}
\nc{\tF}{\widetilde{F}}
\nc{\tG}{\widetilde{G}}
\nc{\tf}{\widetilde{f}}
\nc{\Com}{\on{Com}}

\nc{\G}{\mathbb{G}}
\nc{\cG}{\mathcal{G}}
\nc{\cE}{\mathcal E}
\nc{\cF}{\mathcal F}
\nc{\cR}{\mathcal R}
\nc{\cD}{\mathcal D}
\nc{\cB}{\mathcal B}
\nc{\cT}{\mathcal T}
\nc{\cL}{\mathcal L}

\nc{\bM}{\mathbf M}
\nc{\bN}{\mathbf N}
\nc{\U}{\mathbf U}
\nc{\BM}{\mathbf B \mathbf M}
\nc{\Dsg}{\on{D}_{\on{sg}}}
\nc{\fC}{\mathcal{C}}
\nc{\fG}{\mathcal{G}}
\nc{\N}{\mathbb{N}}

\nc{\del}{\partial}
\nc{\cone}{\on{cone}}
\nc{\D}{\on{D}_{\on{diff}}}
\nc{\DMb}{\on{D}^b_{\DM}}
\nc{\Db}{\on{D}^{\on{b}}}
\nc{\Kb}{\on{K}^{\on{b}}}
\nc{\fm}{\mathfrak{m}}
\nc{\Flag}{\on{Flag}}
\nc{\DMmin}{\DM_{\on{min}}}
\nc{\Ddiff}{\on{D}_{\on{diff}}}
\nc{\Dbdiff}{\on{D}^\on{b}_{\on{diff}}}
\nc{\wO}{\widehat{\OO}}
\nc{\wT}{\widehat{T}}
\nc{\from}{\leftarrow}
\nc{\wLL}{\widetilde{\LL}}
\nc{\augCech}{\widetilde{\cC}}
\nc{\Fold}{\on{Fold}}
\nc{\Ext}{\on{Ext}}
\nc{\FF}{\mathbf{F}}
\nc{\Comper}{\Com_{\on{per}}}
\nc{\Unfold}{\on{Unfold}}
\nc{\intHom}{\underline{\Hom}}
\nc{\Ex}{\on{Ex}}
\nc{\tg}{\widetilde{g}}

\def\b{\beta}
\nc{\B}{\mathcal{B}}
\nc{\K}{\mathcal{K}}
\nc{\kos}{\on{Kos}}
\nc{\Perf}{\on{Perf}}
\nc{\tR}{\widetilde{\cR}}
\nc{\X}{\mathcal{X}}
\nc{\Cl}{\on{Cl}}
\nc{\fU}{\mathcal{U}}
\nc{\tA}{\widetilde{A}}
\nc{\tB}{\widetilde{B}}
\nc{\tC}{\widetilde{C}}

\title{Minimal free resolutions of differential modules}
\author{Michael K. Brown}
\author{Daniel Erman}

\date{\today}

\begin{document}

\maketitle
\begin{abstract}
We propose a notion of minimal free resolutions for differential modules, and we prove existence and uniqueness results for such resolutions. We also take the first steps toward studying the structure of minimal free resolutions of differential modules. Our main result in this direction explains a sense in which the minimal free resolution of a differential module is a deformation of the minimal free resolution of its homology; this leads to structural results that mirror classical theorems about minimal free resolutions of modules.

\end{abstract}
%\tableofcontents

%%%%%%%%%%%%%%%%%%%%%%
\section{Introduction}
%%%%%%%%%%%%%%%%%%%%%%
A \defi{differential module} is a module $D$ equipped with an endomorphism $\del\colon D\to D$ that squares to 0. 
Differential modules generalize chain complexes and appear in the literature as far back as Cartan-Eilenberg's text \cite{CE}.  The study of differential modules was truly launched by Avramov-Buchweitz-Iyengar's seminal work \cite{ABI}, and they have played an important role in commutative algebra and representation theory in recent years, for instance in articles by Iyengar-Walker \cite{IW}, Ringel-Zhang \cite{RZ}, and Rouquier \cite{rouquier}, among others, e.g. \cite{devries, stai2, stai, su,th,wei,xyy}. %(see Example~\ref{fold})

The central theme of Avramov-Buchweitz-Iyengar's article \cite{ABI} is that theorems about free complexes ought to have analogues for what they call free differential flags, as these differential modules come equipped with a filtration that can act as a substitute for a homological grading.  With this in mind, we ask whether one can find an analogue of minimal free resolutions for differential modules.

\iffalse
More specifically, this paper has two main goals. First, we develop a definition of a minimal free resolution of a differential module--which turns out to be somewhat subtle--and we prove existence and uniqueness results for such resolutions.
Second, we take the first steps toward studying the structure of minimal free resolutions of differential modules. In particular, in Theorem~\ref{thm:deformation}, we show that the minimal free resolution of a finitely generated differential module is
a deformation of the minimal free resolution of its homology; this allows one to transfer numerical and structural results about minimal free resolutions of modules to the setting of differential modules. 
\fi

\subsection{Motivation}
\label{motivation}
A main source of our motivation for this project comes from our work on sheaves over toric varieties in \cite{Tate}. In more detail: the main goal of \cite{Tate} is to extend Eisenbud-Fl\o ystad-Schreyer's theory of Tate resolutions of sheaves on $\PP^n$ \cite{EFS} to more general toric varieties. The results in \cite{Tate} require a multigraded version of the BGG correspondence, which gives an adjunction between complexes of modules over the Cox ring of a projective toric variety and differential modules over its Koszul dual exterior algebra $E$ (see \cite{HHW} for a proof of the multigraded BGG correspondence, or see \cite{Tate} for an overview). Just as the theory of Tate resolutions in \cite{EFS} involves minimal free resolutions of $E$-modules, our toric generalization involves a theory of minimal free resolutions for differential $E$-modules. %\michael{I commented out the last sentence of this paragraph involving the sheaf cohomology algorithm, since we no longer do that in the Tate paper. Actually I think I'd like to hold off on making any specific references to the Tate paper, because possibly the Section/Theorem numbers will change with the next revision}%For instance, we apply the results in this paper to develop a novel algorithm for computing sheaf cohomology on weighted projective stacks; see \cite[Section 6.2]{Tate} for details.

Another source of motivation comes from commutative algebra. A conjecture of Avramov-Buchweitz-Iyengar \cite[Conjecture 5.3]{ABI} predicts that, over a local ring of Krull dimension $d$, there is a lower bound on the rank of certain free differential modules with nonzero finite length homology that depends on $d$. This conjecture generalizes a long-standing conjecture of Carlsson in algebraic topology \cite[Conjecture I.3]{carlsson} and a folklore conjecture of Avramov about free complexes with nonzero finite length homology (see the introduction of~\cite{IW}); it is also closely related to conjectures of Buchsbaum--Eisenbud~\cite[p. 453]{buchs-eis-gor} and Horrocks~\cite[Problem 24]{hartshorne-vector} on the ranks of free modules in a minimal free resolution of a finite length module. While the Avramov-Buchweitz-Iyengar conjecture was recently proven by Iyengar-Walker to be false \cite{IW}, the conjecture and its consequences nevertheless highlight the importance of understanding homological properties of free differential modules. We hope that a notion of minimal free resolutions of differential modules will provide a new and fruitful perspective on this topic; see the end of Section \ref{results} for more details.

\subsection{Results}
\label{results}
Let $R$ be graded local ring; that is, assume $R = \bigoplus_{i = 0}^\infty R_i$ is a Noetherian $\N$-graded ring with $R_0$ local. How should one define a minimal free resolution of a differential module $D$?  A first guess might be to define it as a quasi-isomorphism $F \xra{\simeq} D$,
where $F$ is free, and the differential $\del_F$ on $F$ is \defi{minimal}, i.e. $\del_F(F) \subseteq \m F$. The problem with this definition is that such resolutions are not unique up to isomorphism.  For instance, let $R = \QQ[x]/(x^2)$, and take $D=0$. For any $n > 0$, we can take $F=R^n$, with differential given by $\del_F = x\cdot \id_F$; the map $(R^n,\del_F) \to (0,0)$ is a quasi-isomorphism for all $n$.

The problem here is already present in the study of free resolutions of complexes of modules: a pair of minimal free resolutions $F, F'$ of a chain complex $C$ are only guaranteed to be isomorphic if they are bounded below, i.e. $F_i = 0 = F_i'$ for $i \ll 0$ \cite[Proposition 4.4.1]{roberts}. The above example may be reinterpreted as a collection of non-isomorphic minimal free complexes 
$$
\cdots \xra{x} R^n \xra{x} R^n \xra{x} \cdots,
$$
each of which is a free resolution of the 0 complex.

We are then led to ask: what is the right notion of ``bounded below'' in the setting of differential modules? Avramov-Buchweitz-Iyengar's work points toward the answer: a \defi{free flag} is a differential module of the form $F = \bigoplus_{i \ge 0} F_i$, where each $F_i$ is free, and $\del_F(F_i) \subseteq \bigoplus_{j < i} F_j$ (\cite{ABI} Section 2).\footnote{While \cite{ABI} do not equip their flags with splittings $F=\bigoplus_{i\ge 0} F_i$ of the filtration, we include the choice of a splitting to streamline presentations of differentials as matrices.} We define a \defi{free flag resolution} of a differential $R$-module $D$ to be a quasi-isomorphism $F \xra{\simeq} D$, where $F$ is a free flag. 

It follows from the proof of \cite[Lemma 3.5]{stai} that every differential $R$-module admits a free flag resolution; we discuss two ways to construct free flag resolutions in Section \ref{flagressection}.
We can represent the differential of a free flag $(F,\del_F)$ as a matrix of the form
$$
\bordermatrix{
	&F_0&F_{1}&F_{2}&F_{3}&\cdots &F_n &\cdots \cr
F_{0}&0&\partial_{1,0}&\partial_{2,0}&\partial_{3,0}&\cdots &\partial_{n,0}&\cdots \cr
F_{1}&0&0&\partial_{2,1}&\partial_{2,0}&\cdots &\partial_{n,1}&\cdots \cr
\ \ \vdots &\vdots&\vdots&\vdots&\vdots&\ddots &\vdots&\ddots \cr
F_{n-1}&0&0&0&0&\cdots &\del_{n,n-1}&\cdots \cr
F_n&0&0&0&0&\cdots &0&\cdots \cr	
\ \ \vdots &\vdots&\vdots&\vdots&\vdots&\ddots &\vdots&\ddots \cr
},
$$
where each $\del_{i,j}$ is a map of free modules. Notice that, if $F$ is a free flag such that $\del(F_i) \subseteq F_{i - 1}$ for all $i \ge 1$, then $F$ can be viewed as a bounded below complex of free modules. 

Since flags provide an analogue of bounded below complexes, a second guess might be that a minimal free resolution of a differential $R$-module $D$ should be a free flag resolution 
$
F \xra{\simeq} D
$
such that $\del_F$ is minimal. Unfortunately, this definition has the opposite problem: such resolutions don't always exist.

\begin{example}
\label{notflag}
Let $R = \QQ[x, y]$, with the standard grading.  Take $D = R^{\oplus 2}$, with differential
\begin{equation}
\label{differential}
\del_D = \begin{pmatrix}
xy & -x^2 \\
y^2 & -xy\\
\end{pmatrix}.
\end{equation}
While $D$ is free and minimal, it is observed in \cite{ABI} that it does not admit the structure of a free flag (one can check this directly or see this from \cite[Rank Inequality]{ABI}, as in \cite[Example 5.6]{ABI}). It follows that $D$ is not a free flag resolution of itself.  In fact, we will show in Example \ref{DNE} that $D$ does not admit any free flag resolution whose differential is minimal.
\end{example}

However, it turns out that the differential module $D$ in Example \ref{notflag} is a \defi{summand} of the (non-minimal) free flag resolution $F = R^{\oplus 4}$ of $D$ with differential
$$
\del_F = \begin{pmatrix} 0 & -y & -x & -1 \\ 0 & 0 & 0 & -x \\ 0 & 0 & 0 & y \\ 0 & 0 & 0 & 0 \end{pmatrix};
$$
see Example~\ref{DNE} for details. In particular, summands of free flags need not be free flags. This finally leads to our definition:
%A minimal A differential module $N$ is an \defi{essential summand of $M$} if there is a splitting $M=N\oplus T$ where $T$ is contractible.  
a \defi{free resolution} of a differential module $D$ is a quasi-isomorphism
$$
\epsilon\colon F \xra{\simeq} D
$$
such that $F$ is a free differential module, and $\epsilon$ factors as $F \to \tF \to D$, where $\tF \to D$ is a free flag resolution, and $F \to \tF$ is a split injection of differential modules. We say $\epsilon\colon F \xra{\simeq} D$ is a \defi{minimal free resolution} if $\del_F$ is minimal. %that is, if $\del_F(F)\subseteq \mathfrak m F$.

For instance, the differential module $D$ in Example \ref{notflag} is a minimal free resolution of itself. For an example that is minimal and free but is not a minimal free resolution of itself, consider again the examples of $R = \QQ[x]/(x^2)$, $F = R^n$ for $n > 0$, and $\del_F = x \cdot \id_F$. Since each such $F$ is acyclic but not contractible, it cannot be a summand of a flag by \cite[Theorem 3.5]{ABI}.

To motivate our definition of a free resolution, we observe that an aspect of free flags that makes them suitable to play the role of free resolutions is a certain lifting property: given a pair of free flag resolutions $F \to D$ and $F' \to D'$ and a morphism $f \colon D \to D'$ of differential modules, there is a morphism $\tf\colon F \to F'$ that makes the square
$$
\xymatrix{
F \ar[r] \ar[d]^-{\tf} & D \ar[d]^-{f} \\
F' \ar[r] & D'
}
$$
commute up to homotopy, and the lift $\tf$ is unique up to homotopy (see Lemma~\ref{prop:lifting}). But \emph{summands} of free flag resolutions also have this lifting property, so it is 
%{natural to expect}
reasonable to also use them to play the role of free resolutions. This subtlety is not present in the study of free resolutions of complexes, because while a summand of a free flag need not be a free flag, a summand of a bounded below complex of free modules is again a bounded below complex of free modules (recall that $R$ is graded local, so all projectives are free).

We prove the following existence and uniqueness results for minimal free resolutions of differential modules; see Theorem~\ref{technical} below for a slightly more general statement. We note that, throughout the paper, we denote the homology of a differential module $D$ by $H(D)$.

\begin{thm}
\label{thm:exist}
Let $D$ be a differential $R$-module such that $H(D)$ is finitely generated.
\begin{itemize}
\item[(a)] If $H(D)$ has finite projective dimension, then $D$ admits a minimal free resolution $M\to D$ that is unique up to isomorphism of differential modules and such that $M$ has finite rank as an $R$-module.
\item [(b)] If $R_0$ is a field, and the differential on $D$ is homogeneous of degree 0, then $D$ admits a minimal free resolution $M\to D$ that is unique up to isomorphism of differential modules. Any basis for $M$ as an $R$-module has finitely many elements of degree $j$ for any $j\in \ZZ$ and no degree $j$ elements for $j\ll 0$.
\end{itemize}
\end{thm}

\begin{remarks}\label{remarks}
\text{ }
\begin{enumerate}

\item The hypotheses in Theorem \ref{thm:exist} provide finiteness conditions that are needed to apply Nakayama's Lemma to $M$.  See also Remark~\ref{difficulty}.  

\item If $D$ is finitely generated with zero differential, then the minimal free resolution of $D$ is just the usual minimal free resolution of the underlying module, considered as a differential module.

\item Passing from a free flag resolution to a minimal free resolution provides uniqueness at the potential cost of the flag structure. In fact, minimal free resolutions have genuinely distinct behavior from free flags.  For instance, the minimal free resolution $D$ in Example~\ref{notflag} does not satisfy~\cite[Rank Inequalities]{ABI}. 
However, Theorem~\ref{thm:deformation} gives a precise description of the relationship between the minimal free resolution of $D$ and the minimal free resolution of its homology, in the spirit of the above theorem of Avramov-Buchweitz-Iyengar.

\item One can define the Betti numbers $\beta_j^{\DM}(D)$ of a differential module $D$ in terms of Tor, just as in classical homological algebra (see Section \ref{betti}). If $M$ is a minimal free resolution of $D$, then $\beta_j^{\DM}(D)$ is the number of degree $j$ elements in any basis of $M$. 

\end{enumerate}
\end{remarks}

Our second main result shows that the minimal free resolution of a differential module is closely related to--and is, in fact, a deformation of--the minimal free resolution of its homology, providing the beginning of a structure theory for such resolutions. Here we state the result for degree $0$ differential modules; see Theorem~\ref{thm:deformation a} for the more general statement.

\begin{thm}\label{thm:deformation}
Let $D$ be a differential $R$-module whose differential is homogeneous of degree $0$ and such that $\HH(D)$ is finitely generated. 
Let
\[
F_0 \overset{\del_{1,0}}{\longleftarrow}F_1 \overset{\del_{2,1}}{\longleftarrow}F_2 \overset{\del_{3,2}}{\longleftarrow}\cdots
\]
be a minimal free resolution of $\HH(D)$.  There exists a free flag resolution of $D$ such that the underlying module is $\bigoplus_{i} F_i$, and the differential has the form:
\[
\del=\bordermatrix{
	&F_0&F_{1}&F_{2}&F_{3}&\cdots &F_n &\cdots \cr
F_{0}&0&\del_{1,0}&\del_{2,0}&\del_{3,0}&\cdots &\del_{n,0}&\cdots \cr
F_{1}&0&0&\del_{2,1}&\del_{3,1}&\cdots &\del_{n,1}&\cdots \cr
\ \ \vdots &\vdots&\vdots&\vdots&\vdots&\ddots &\vdots&\ddots \cr
F_{n-1}&0&0&0&0&\cdots &\del_{n,n-1}&\cdots \cr
F_n&0&0&0&0&\cdots &0&\cdots \cr	
\ \ \vdots &\vdots&\vdots&\vdots&\vdots&\ddots &\vdots&\ddots \cr
}.
\]
\end{thm}

The $\del_{i,j}$ with $i-j>1$ measure, in a sense, the extent to which $D$ fails to be quasi-isomorphic to its homology.
In fact, one can use this result to produce an explicit degeneration (see Remark~\ref{rmk:degeneration}) from a resolution of the differential module $D$ to a resolution of the ordinary module $\HH(D)$.
This allows one to transfer structural results on minimal free resolutions from the category of modules to that of differential modules.  For instance, we obtain an analogue of the Hilbert-Burch theorem for differential modules:
\begin{cor}\label{cor:HB}
Let $R=k[x_1,\dots,x_n]$, and let $D$ be a differential $R$-module whose differential is homogeneous of degree 0. Assume that $\HH(D)$ is a Cohen-Macaulay, codimension $2$ quotient of $R$.  The minimal free resolution of $D$ is of the form
\[
\bordermatrix{
&R&F_1&F_2\cr
R&0&\del_{1,0}&\del_{2,0}\cr
F_1&0&0&\del_{2,1}\cr
F_2&0&0&0
},
\]
where, setting $r = \rank F_2$, we have $\rank F_1 = r+1$, and the entries of $\del_{1,0}$ are the $r\times r$ minors of $\del_{2,1}$.\end{cor}

%We now return to the rank conjectures mentioned above.  
As noted earlier, Avramov-Buchweitz-Iyengar conjectured in \cite{ABI} that, if $R$ is a local ring, and $F$ is a free flag differential module with nonzero finite length homology, then $\on{rank} F \ge 2^{\dim R}$. This was largely motivated by similar conjectures of 
Buchsbaum-Eisenbud ~\cite[p. 453]{buchs-eis-gor} and Horrocks~\cite[Problem 24]{hartshorne-vector} about the minimal rank of a free resolution of a nonzero finite length module. Walker recently proved a variant of these conjectures for minimal free resolutions \cite{W}, while Iyengar-Walker provided counterexamples to Avramov-Buchweitz-Iyengar's conjecture for flag differential modules \cite{IW}.
Theorem~\ref{thm:deformation} offers a new perspective on the tension between these recent results, providing a comparison between the minimal free resolution of $D$ and that of $\HH(D)$, and showing that the rank of the minimal free resolution of $D$ can only be strictly smaller than the rank of the minimal free resolution of $H(D)$ when the maps $\del_{i,j}$ for $i - j > 1$ in Theorem~\ref{thm:deformation} are not minimal. This raises a new question:
\begin{question}
\label{question}
Let $R$ be a graded local ring of Krull dimension $d$, and let $D$ be a differential $R$-module such that $H(D)$ is nonzero and has finite projective dimension and finite length. Let $M$ be the minimal free resolution of $D$. What is a maximal lower bound for $\on{rank}(M)$ in terms of $d$?
\end{question}

Iyengar and Walker produce examples in \cite{IW} where $H(M)=k^{\oplus 2}$ and where, for $d$ even, $\rank M= \binom{d+2}{\frac{d}{2}+1}$.  Thus, any bound in Question \ref{question} must be significantly smaller than $2^d$ (see~\cite[Remark 4.9]{IW}). For instance, recent work of Banks-VandeBogert provides examples where $H(M)=k$ and where $\rank M = 2^{d-1}$~\cite[Proposition 3.8]{banks-keller}.

The paper is organized as follows.  Section~\ref{sec:background} provides background on differential modules, including a summary of how to construct free flag resolutions.  We prove Theorem~\ref{thm:deformation} in Section \ref{sec:deformation}. In Section~\ref{minsection}, we introduce a minimization procedure for differential modules and use it to prove Theorem~\ref{thm:exist}.  Section~\ref{sec:examples and applications} contains a number of applications of our main results to the study of Betti numbers, Cohen-Macaulay codimension two and Gorenstein codimension three examples, and more.

\subsection*{Acknowledgements} We thank Maya Banks, David Eisenbud, Srikanth Iyengar, and Frank-Olaf Schreyer for helpful conversations, and we thank the referee for their careful reading and useful comments. The computer algebra system {\em Macaulay2}~\cite{M2} provided valuable assistance throughout our work.

\subsection*{Conventions}  
Throughout this paper, $R$ denotes a \defi{graded local ring}, by which we mean an $\mathbb N$-graded Noetherian (not necessarily commutative) ring $R=\bigoplus_{i \ge 0} R_i$, where $R_0$ is (not necessarily commutative) local ring.\footnote{The Noetherian and local assumptions are not necessary in \S \ref{sec:background}, except in the remarks concerning finitely generated/minimal free flag resolutions in Constructions \ref{const:stai} and \ref{killing cycles}.}  The arguments for graded local commutative rings go through verbatim without the commutative hypothesis, so there is no cost to working in this generality.  Also, as noted above, we are partly motivated by applications involving exterior algebras.  All $R$-modules are left modules; moreover, all $R$-modules and morphisms of $R$-modules are graded, though we may sometimes omit this hypothesis to avoid tedious repetitions.

%%%%%%%%%%%%%%%%%%%%%%
\section{Background on differential modules}\label{sec:background}
%%%%%%%%%%%%%%%%%%%%%%
%%%%%%%%%%%%%%%%%%%%%%
\subsection{Basic definitions}
%%%%%%%%%%%%%%%%%%%%%%

Let $a \in \ZZ$.  A \defi{degree $a$ differential $R$-module} is a pair $(D, \del)$, where $D$ is an $R$-module, and $\del\colon D \to D(a)$ is a homogeneous, $R$-linear map such that $\del^2 = 0$. 
A \defi{morphism} $D \to D'$ of degree $a$ differential modules is a homogeneous map $f \colon D \to D'$ satisfying $f \del = \del'  f$.

We write $\DM(R,a)$ for the category of differential modules of degree $a$. The \defi{homology} of an object $D \in \DM(R,a)$ is the subquotient
$$
\ker(\del \colon D \to D(a) / \im(\del \colon D(-a) \to D),
$$
denoted $\HH (D)$. A morphism in $\DM(R, a)$ is a \defi{quasi-isomorphism} if it induces an isomorphism on homology. A \defi{homotopy} of morphisms $f, f' \colon D \to D'$ in $\DM(R, a)$ is a morphism $h \colon D\to D'(-a)$ of $R$-modules such that $f - f' = h \del + \del' h$. A differential module $D \in \DM(R, a)$ is \defi{contractible} if $\id_D$ is homotopic to 0. The \defi{mapping cone} of a morphism $f \colon D \to D'$ in $\DM(R, a)$ is the module $D' \oplus D(a)$ equipped with the differential
\[
\bordermatrix{
&D'&D(a)\cr
D'(a)& \del' & f \cr
 D(2a)& 0 & -\del 
 }.
 \]

\begin{prop}
\label{contractible}
An object in $\DM(R, a)$ is contractible if and only if it is isomorphic to an object of the form
\begin{equation}
\label{conobject}
M \oplus M(a) \xra{\begin{pmatrix} 0 & 1 \\ 0 & 0 \end{pmatrix}} M(a) \oplus M(2a)
\end{equation}
for some $R$-module $M$.
\end{prop}

\begin{proof}
An object of the form \eqref{conobject} may be equipped with the contracting homotopy $\begin{pmatrix} 0 & 0 \\ 1 & 0 \end{pmatrix}$. Conversely, suppose $(D, \del) \in \DM(R, a)$ is contractible. Choose a contracting homotopy $h : D \to D(-a)$. We first observe that, if $z \in D$ is a cycle, then $z = (\del h)(z)$; this implies that $H(D) = 0$. 
Let $Z$ denote the submodule of cycles in $D$, and consider $Z \oplus Z(a)$ as an object in $\DM(R, a)$ with differential $\begin{pmatrix} 0 & 1 \\ 0 & 0 \end{pmatrix}$. The map $D \to Z \oplus Z(a)$ given by the transpose of the matrix $\begin{pmatrix} 1 - h\del & \del \end{pmatrix}$ is an isomorphism of differential modules with inverse $\begin{pmatrix} 1 & h\end{pmatrix}$.
\end{proof}

\begin{example}
\label{fold}
Any complex $C = ( \cdots \to C_i \xra{\del_i} C_{i-1} \to \cdots)$ of $R$-modules determines a differential $R$-module, in the following way. We can ``fold" $C$ into an object in $\DM(R, a)$, for any $a \in A$; that is, we take $D = \bigoplus_{i} C_i(ia)$ and $\del \colon D \to D(a)$ to be the map induced by the $\del_i$. This gives a functor
$$
\Fold\colon \on{Com}(R) \to \DM(R, a)
$$
for all $a$. For example, let $R=k[x,y]$, and take
$$
K = \left( R  \xla{\begin{pmatrix} x & y\end{pmatrix}} R(-1)^{\oplus 2} \xla{\begin{pmatrix} -y \\ x \end{pmatrix}} R(-2)\right),
$$
the Koszul complex on $x$ and $y$, concentrated in homological degrees $0,1$, and 2. The object $\Fold(K) \in \DM(R,a)$ has underlying module
$$
D = R \oplus R(a-1)^{\oplus 2} \oplus R(2a-2) 
$$ 
and differential
$$
\del = \bordermatrix{
&R&R(a-1)&R(a-1)&R(2a-2)\cr
R(a)&0&x&y&0 \cr
R(2a-1)&0&0&0&-y \cr
R(2a-1)&0&0&0&x \cr
R(3a-2)&0&0&0&0
}.
$$
\end{example}

\begin{remark}
The category $\DM(R, a)$ is equivalent to the category of left dg-modules over the dg-algebra $R[x, x^{-1}]$ with trivial differential, where $x$ has homological degree $-1$ and internal degree $a$. The equivalence sends an object $(D, \del_D)$ to the complex 
$$
\cdots \xra{-x \del_D} D \xra{x \del_D} x D \xra{-x \del_D} \cdots
$$
with the obvious $R[x, x^{-1}]$-action. In particular, it follows from well-known results about dg-modules that $\DM(R, a)$ is an abelian category. The above notions of quasi-isomorphism, homotopy, and mapping cone for differential modules all correspond, via this equivalence, to the usual notions for complexes.
\end{remark}

\begin{remark}\label{rmk:a matters}
The categories $\DM(R, a)$ change as $a$ varies. For instance, suppose $R = \Z[x]/(x^2)$, where $|x| = 1$. The category $\DM(R, 1)$ contains a rank 1 free differential module that is exact but not contractible, namely $R \overset{x}{\longrightarrow} R(1)$. But, when $a \ne 1$, $\DM(R, a)$ contains no such object.

This marks a departure from the setting of chain complexes. If $M \to M'$ is an $R$-linear map that is homogeneous of degree $a$, one can twist the source or target to form an equivalent map that is homogeneous of degree 0; for this reason, it suffices to study complexes of graded modules whose differentials are homogeneous of degree $0$. But, as the previous example illustrates, this trick doesn't work in the setting of differential modules.
\end{remark}

\begin{remark}\label{rmk:other gradings}
If $R$ is graded by an arbitrary abelian group $A$, then one can define the category $\DM(R,a)$ in the same way as above.  As long as a version of Nakayama's Lemma holds for the $A$-graded ring $R$, one can apply standard techniques to extend all of the results in this paper to this setting.  For instance, our results extend to the case where $R=k[x_1,\dots,x_n]$ is equipped with the monomial $\mathbb N^n$-multigrading, which is the focus of~\cite{devries}.
\end{remark}
%
%%%%%%%%%%%%%%%%%
\subsection{Free flag resolutions}
\label{flagressection}

Recall the definition of a free flag resolution from Section \ref{results}. The following result is well-known; cf. \cite[Proposition~3.5]{stai}:
\begin{prop}
\label{resexists}
Every $D \in \DM(R, a)$ admits a free flag resolution. 
\end{prop}

We give two proofs of Proposition \ref{resexists} (Constructions \ref{const:stai} and \ref{killing cycles}), as both constructions will be useful in this paper.

\begin{construction}\label{const:stai}
This construction uses a Cartan-Eilenberg-type resolution and appears in \cite[Lemma~3.5]{stai}. The base ring in \cite{stai} is assumed to have finite global dimension and is not graded, but the construction adapts easily to the graded case, and removing the finite global dimension assumption just allows for the resolution to be infinitely generated. 

In detail: let $Z$, $B$, and $H$ denote the cycles, boundaries, and homology of $D$. We have short exact sequences

\[
0 \to B\to Z \to H \to 0
\qquad
\text{ and} 
\qquad
0\to Z \to D \to B(a)\to 0.
\]
Choose a free resolution $F^B$ of $B$, and let $G^B$ be the same complex, but with $G^B_i=F^B_i(ia)$.  Notice that the differential on $G^B$ is homogeneous of degree $a$, with respect to the internal grading. Choose a free resolution $F^H$ of $H$, and define $G^H$ similarly. We obtain a differential module $G:=G^B\oplus G^H\oplus G^{B(a)}$ whose differential
decomposes as $\del+\epsilon$, where:
\begin{itemize}
	\item $\del\colon G\to G(a)$ is the degree $a$ folding (see Example \ref{fold}) of a free resolution of $D$ obtained by applying the Horseshoe Lemma to the two above exact sequences, and
	\item $\epsilon\colon G\to G(a)$ sends the third summand $G^{B(a)}$ of $G$ to the first summand $G^{B}(a)$ of $G(a)$ via the identity.
\end{itemize}
One can check that a flag structure on $G$ may given by
\[
G_i =G^B_i \oplus G^H_i \oplus G^{B(a)}_{i-1}.
\]
Indeed, $(\del+\epsilon)\colon G\to G(a)$ satisfies $(\del+\epsilon)(G_i)\subseteq G_{i-1}(a)\oplus G_{i-2}(a)$.

The rest of Stai's argument goes through essentially verbatim. Note that, when $D$ is finitely generated, we can choose $F^B$ and $F^H$ so that they're minimal, but the resulting free flag resolution will still typically not be minimal. When $F^B$ and $F^H$ are finite free resolutions whose terms are finitely generated, the resulting free flag resolution is finitely generated.\end{construction}

\begin{construction}\label{killing cycles}
Our second construction mimics the usual algorithm for resolving a module. Choose a set of cycles in $D$ that descends to a homogeneous generating set of $H(D)$. Let $F_0$ be a free $R$-module with basis indexed by this set, where the basis element corresponding to a cycle $c$ is in degree $|c|$. Regard $F_0$ as a differential module with trivial differential, and let $\epsilon_0\colon F_0 \to D$ be the morphism of differential modules that sends each basis element to its associated cycle. Next, choose a set of cycles in $\cone(\epsilon_0)$ that descends to a homogeneous generating set of $H(\cone(\epsilon_0))$. Construct $F_1$ in the same way as $F_0$, and define $\epsilon_1 \colon F_1 \to \cone(\epsilon_0)$ in the same way as well. Notice there is a natural map $\on{cone}(\epsilon_0) \to \on{cone}(\epsilon_1)$. Continuing in this way, we get a sequence
$$
\on{cone}(\epsilon_0) \to \on{cone}(\epsilon_1) \to \on{cone}(\epsilon_2) \to \cdots
$$
of differential modules. Taking the colimit, we obtain a flag $F = \bigoplus_{i \ge 0} F_i$ and a quasi-isomorphism $\epsilon \colon F \xra{\simeq} D$. If $H(D, \del)$ is finitely generated, then we can choose minimal generating sets in every step of this process, yielding a free flag resolution where $F_i$ is finitely generated for all $i$.
\end{construction}

\begin{remark}
Constructions~\ref{const:stai} and \ref{killing cycles} can both be implemented in a computer algebra package.  
%Construction~\ref{const:stai} essentially involves computing minimal free resolutions $F^H$ and $F^B$ and then performing matrix division to construct the maps arising from the Horseshoe Lemma. Construction~\ref{killing cycles} can be implemented by a minor variant of the algorithm for computing a minimal free resolution.  
It would be interesting to understand which of these constructions--or perhaps even some other construction--is the most efficient.
\end{remark}

\vskip\baselineskip

\begin{remark}
Just like free resolutions of modules, free flag resolutions of differential modules are not unique. But there is an additional choice involved when constructing a free flag resolution: not only are the differential and underlying module of a free flag resolution not unique, but, once these are fixed, the flag structure is not unique. 

For instance, take $R=k[x,y]$, where $k$ is a field, and let $D$ be the differential module $k\overset{0}{\to} k$ in $\DM(R,0)$. A natural choice for a free flag resolution $F \xra{\simeq} D$ is to set $F_0=R, F_1=R(-1)^{\oplus 2}$, and $F_2=R(-2)$; and to equip $F = F_0 \oplus F_1 \oplus F_2$ with the Koszul differential.  But another flag resolution is given by setting $F_0=R, F_1=R(-1), F_2=R(-1)$ and $F_3=R(-2)$ with differential
\[
\xymatrix{
R(-2) \ar[r]^-{-y}\ar@/_1.8pc/[rr]_-{x}& R(-1)\ar[r]^-{0}\ar@/_-1.8pc/[rr]^-{x}& R(-1)\ar[r]^-{y}&R.
}
\]
There are infinitely many other choices for the flag structure on this resolution.
\end{remark}

We now turn to the lifting property of free flag resolutions.  In classical homological algebra, maps of modules may be lifted to maps on free resolutions. A similar statement holds for differential modules: given a diagram
\begin{equation}\label{eqn:lifting}
\xymatrix{
F  \ar[r]^{\epsilon}& D\ar[d]^f\\
F' \ar[r]^{\epsilon'}& D'
}
\end{equation}
in $\DM(R, a)$, where $\epsilon$ and $\epsilon'$ are free flag resolutions, there is a map $\tf\colon F \to F'$ that makes the diagram commute, up to homotopy. More generally, we have the following result. The proof is similar to the proof of the analogous statement for complexes.

\begin{lemma}[cf. \cite{FHT} Proposition 6.4] \label{prop:lifting}
Suppose we have a diagram of differential $R$-modules of degree $a$, as in \eqref{eqn:lifting}, 
where $F$ is a summand of a free flag, and $\epsilon'$ is a quasi-isomorphism. There is a morphism $\tf\colon  F \to F'$ that makes \eqref{eqn:lifting} commute up to homotopy, and $\tf$ is unique up to homotopy.
\end{lemma}

\begin{remark}\label{rmk:functorial}
It follows immediately from the definition that free resolutions of differential modules inherit the lifting property of Lemma~\ref{prop:lifting}.
\end{remark}

%%%%%%%%%%%%%%%%%%%%%%%%%%%%%%%
\section{Degeneration to the homology}\label{sec:deformation}
%%%%%%%%%%%%%%%%%%%%%%%%%%%%%%%
In this section, we prove Theorem~\ref{thm:deformation}.  We will use the following elementary fact: if $(F,\del)$ is a differential module, and $Q\colon F\to F$ is an automorphism of $F$, then $(F,Q^{-1}\del Q)$ is isomorphic to $(F,\del)$ as a differential module.
\begin{remark}\label{rmk:degeneration}
Here is the sense in which the minimal free resolution of $\HH(D)$ is a degeneration of the minimal free resolution of $D$.  Consider the differential appearing in Theorem~\ref{thm:deformation}.  Multiplying $\del_{i,j}$ by $t^{i-j-1}$ yields a family of free differential modules over $R[t]$ that are quasi-isomorphic to $D$ for $t\ne 0$ and quasi-isomorphic to $\HH(D)\overset{0}{\to}\HH(D)$ for $t=0$. 
%While it is possible that $F$ is not itself a minimal free resolution of $D$, as the matrices $\del_{i,j}$ with $i-j>1$ might involve units, this nevertheless shows that the minimal free resolutions of $D$ and $\HH(D)$ are closely connected.
\end{remark}

We now prove the following generalization of Theorem~\ref{thm:deformation}, allowing $a\ne 0$, which complicates the notation somewhat:
\begin{thm}\label{thm:deformation a}
Let $D\in \DM(R,a)$, and assume that $\HH(D)$ is finitely generated. Let
\[
F_0 \overset{\del_{1,0}}{\longleftarrow}F_1 \overset{\del_{2,1}}{\longleftarrow}F_2 \overset{\del_{3,2}}{\longleftarrow}\cdots
\]
be a minimal free resolution of $\HH(D)$, twisted so that all maps are homogeneous of degree $a$.\footnote{Thus, replacing $F_i$ by $F_i(-ia)$ and considering the same maps would yield a minimal free resolution of $\HH(D)$ that is homogeneous of degree $0$.  See Remark~\ref{rmk:a matters}.} There exists free flag resolution of $D$, where the underlying module is $\bigoplus_{i} F_i$ and the differential has the form:
\[
\del=\bordermatrix{
	&F_0&F_{1}&F_{2}&F_{3}&\cdots &F_n &\cdots \cr
F_{0}(a)&0&\del_{1,0}&\del_{2,0}&\del_{3,0}&\cdots &\del_{n,0}&\cdots \cr
F_{1}(a)&0&0&\del_{2,1}&\del_{3,1}&\cdots &\del_{n,1}&\cdots \cr
\ \ \vdots &\vdots&\vdots&\vdots&\vdots&\ddots &\vdots&\ddots \cr
F_{n-1}(a)&0&0&0&0&\cdots &\del_{n,n-1}&\cdots \cr
F_n(a)&0&0&0&0&\cdots &0&\cdots \cr	
\ \ \vdots &\vdots&\vdots&\vdots&\vdots&\ddots &\vdots&\ddots \cr
}.
\]
\end{thm}

\begin{proof}
Apply Construction \ref{const:stai} to get a free flag resolution of $D$ of the form $G^B \oplus G^H \oplus G^B(a)$, with $G^H$ and $G^B$ constructed from minimal free resolutions of the homology and boundaries of $D$, respectively. The differential on this free flag resolution is of the form
$$
\del=\bordermatrix{
&G^B&G^H&G^{B(a)} \cr
G^B(a)&\del^B&\alpha&\id+\gamma \cr
G^H(a)&0&\del^H&\beta\cr
G^{B(a)}(a)&0&0&\del^B
},
$$
where $\del^B$ and $\del^H$ are the differentials on $G^B$ and $G^H$. Here, the map ``$\id$" in the upper-right corner arises from the map $\epsilon$ from Construction \ref{const:stai}. The rest of the matrix arises from the map $\del$ in Construction \ref{const:stai}, i.e. the differential on the free resolution of the underlying module of $D$ obtained by applying the Horseshoe Lemma twice.
%%%%%
\iffalse
We start with the projective resolution construction in~\cite[Lemma~3.5]{stai}.  For simplicity, we focus on the case where $D$ has degree $0$.  To manage the case of degree $a$, one simply twists any module with lower index $i$ by $-ia$, i.e.: $P_i\rightarrow P_i(-ia)$.

Let $B:=\im \del_D$ and $H:=\HH(D)$.  Let $G^B$ and $G^H$ denote the minimal free resolutions of $B$ and $H$ respectively.  Stai's resolution is built from the differentials on $G^B$ and $G^H,$ plus some maps arising from horseshoe lemma.  Let $\alpha\colon G^H[1]\to G^B$ be the anticommutative chain maps arising from the first horsehose
\[
0\to B\to Z\to H\to 0
\]
and let $\gamma\colon G^B[1]\to G^B$ and $\beta\colon G^B[1]\to G^H$ be the anticommutative chain maps arising the second horsehoe
\[
0\to Z\to D\to B\to 0.
\] 
Ignoring the flag structure for the moment, Stai's resolution then has the form:
\[
\del=\bordermatrix{
&G^B&G^H&G^B \cr
G^B&\del^B&\alpha&\id+\gamma \cr
G^H&&\del^H&\beta\cr
G^B&&&\del^B
}.
\]
%where, by 

To be a bit more precise: with the exception of the identity block, all other maps decrease homological indices by $1$, e.g. $\alpha_i\colon G^H_i\to G^B_{i-1}$.  Without the identity block, this would represent the differential of the iterated mapping cone resolution of $D$ itself obtained by the applications of the horseshoe lemma.
\fi
%%%
We set $\tau:=\id+\gamma$. Note that $\gamma$ is nilpotent, since it decreases homological degree and $G^B$ is a bounded below complex. It follows that $\tau$ is invertible, with inverse $\id -\gamma+\gamma^2-\gamma^3 +\cdots$.
%
%Incorporating the identity maps $G^B\to G^B$ yields  the total differential:
%\[
%\bordermatrix{
%&G^B&G^H&G^B \cr
%G^B&\del^B&\alpha&\tau \cr
%G^H&&\del^H&\beta\cr
%G^B&&&\del^B
%}
%\]
We now partially minimize this differential module.  We start by conjugating by the diagonal matrix
$$
\begin{pmatrix}
\id & 0 & 0 \\
0 & \id & 0 \\
0 & 0 & \tau^{-1} \\
\end{pmatrix},
$$
%\[
%P = \bordermatrix{
%&G^B&G^H&G^B \cr
%G^B&\id&& \cr
%G^H&&\id&\cr
%G^B&&&\tau^{-1}
%}
%\]
yielding the differential
\[
\del' =\begin{pmatrix}
\del^B&\alpha&\id \cr
0&\del^H&\beta\tau^{-1}\cr
0&0&\tau\del^B\tau^{-1}
\end{pmatrix}.
%\del' = \bordermatrix{
%&G^B&G^H&G^B(a) \cr
%G^B&\del^B&\alpha&\id \cr
%G^H&0&\del^H&\beta\tau^{-1}\cr
%G^B(a) &0&0&\tau\del^B\tau^{-1}
%}.
\]
Since $(\del')^2 =0$, we have the relations
\begin{equation}
\label{eqn:del'}
\del^B\alpha + \alpha\del^H = 0, \quad \del^B +\alpha\beta\tau^{-1}+\tau\del^B\tau^{-1} = 0, \quad \del^H\beta\tau^{-1}+\beta\del^B\tau^{-1} = 0.
\end{equation}
Next, we let $Q$ be the automorphism
\[
Q = \begin{pmatrix}
\id&0&0\cr
\beta\tau^{-1}&\id&0\cr
\tau\del^B\tau^{-1}&-\alpha&\id\cr
\end{pmatrix};
\]
note that
\[
Q^{-1} =
\begin{pmatrix}
\id&0&0\cr
-\beta\tau^{-1}&\id&0\cr
-\tau\del^B\tau^{-1}-\alpha\beta\tau^{-1}&\alpha&\id\cr
\end{pmatrix}.
\]
A direct computation, using the relations in \eqref{eqn:del'}, shows that
\begin{align*}
Q^{-1}\del' Q 
&=\begin{pmatrix}
0&0&\id\cr
0&\del^H-\beta\tau^{-1}\alpha&0\cr
0&0&0
\end{pmatrix}
\end{align*}
It follows that our free flag resolution is isomorphic, as a differential module, to the direct sum of $(G^H,\del^H - \beta\tau^{-1}\alpha)$ and a contractible differential module.  In particular, there is a quasi-isomorphism $(G^H, \del^H - \beta\tau^{-1}\alpha) \xra{\simeq} D$.
Using that $\tau^{-1}=\id-\gamma+\gamma^2-\gamma^3+\cdots$, we see that the resulting differential has the form:
\[
\bordermatrix{
&G^H_0&G^H_1&G^H_2&G^H_3&G^H_4&\cdots &\cr
G^H_0(a)&0&\del^H_1&-\beta_1\alpha_2&\beta_1\gamma_2\alpha_3&-\beta_1\gamma_2\gamma_3\alpha_4&\cdots &\cr
G^H_1(a)&0&0&\del^H_2&-\beta_2\alpha_3&\beta_2\gamma_3\alpha_4&\cdots &\cr
G^H_2(a)&0&0&0&\del^H_3&-\beta_3\gamma_4&\cdots &\cr
G^H_3(a)&0&0&0&0&0&\cdots &\cr
\vdots &\vdots&\vdots&\vdots&\vdots&\vdots&\ddots \cr
}.
\]
Here, $\a_i$ denotes the restriction of $\a$ to $G_i^H$, and similarly for $\b_i$ and $\gamma_i$. This is clearly a free flag resolution, completing the proof.
\end{proof}

\begin{remark}\label{rmk:deformation cor}
The proof of Theorem~\ref{thm:deformation} gives an explicit formula for the entries of $\del_{i,j}$ in terms of the chain maps 
$\alpha, \beta, \gamma$ arising from the applications of the Horseshoe Lemma.   The entries $\del_{i,i-1}$ are always minimal, since they come from a minimal free resolution of $\HH(D)$.  Every other entry has both an $\alpha$ and a $\beta$ term, and thus if either $\alpha$ or $\beta$ is minimal, then $F$ itself is a minimal free resolution of $D$. On the other hand, if any $\del_{i, j}$ for $i - j > 1$ is not minimal, the minimal resolution of $D$ has smaller rank than the minimal resolution of $H(D)$. This happens, for instance, in Example~\ref{notflag}.
\end{remark}

\begin{remark}
In~\cite{ABI}, the \defi{free class} of a free flag $F$ is defined to be the minimal $\ell$ such that $F$ has a flag structure $F=F_0\oplus F_1\oplus \cdots \oplus F_\ell$.  One can extend this definition to an arbitrary differential module over a local ring by setting the free class of $D$ to be the minimal $\ell$ such that $D$ admits a free flag resolution $F\to D$ where $F$ has free class $\ell$.  Theorem~\ref{thm:deformation} implies that the free class of $D$ is always bounded above by the projective dimension of $\HH(D)$.  It would be interesting to better understand when this inequality is an equality and when it is a strict inequality.  For instance, if $\HH(D)$ is Cohen-Macaulay and has finite projective dimension, then combining Theorem~\ref{thm:deformation} with~\cite[Class Inequality]{ABI} implies that $D$ has free class exactly equal to the projective dimension of $\HH(D)$.
\end{remark}

%%%%%%%%%%%%%%%%%%%%%%%%%%%%%%%
\section{Minimal free resolutions}
\label{minsection}
%%%%%%%%%%%%%%%%%%%%%%%%%%%%%%%
Recall that $R$ is a Noetherian, $\mathbb N$-graded ring $R=\bigoplus_{i\geq 0} R_i$, where $R_0$ is a local ring. We write $\mathfrak m$ for the maximal ideal of $R$, which is the sum of the maximal ideal of $R_0$ and the maximal homogeneous ideal $R_+$.  
We say a morphism $f \colon M \to N$ of $R$-modules is \defi{minimal} if $f(M) \subseteq \m N$.  We write $k$ for $R/\mathfrak m$.

In this section, we prove our existence and uniqueness results for minimal free resolutions.  We first introduce a minimization procedure for finitely generated, free differential modules.  

\begin{prop}\label{prop:minimization procedure}
If $F$ is a finitely generated free differential $R$-module, then $F \cong M \oplus T$, where $M$ is minimal and $T$ is contractible.
\end{prop}
\begin{proof}
If $F$ is minimal, then we are done, so assume otherwise. Choose a basis of $F$, and view $\partial_F$ as a matrix with respect to this basis. We first observe that the condition $\partial_F^2 =0$ forces $\del_F$ to have a unit entry $u$ that does not lie on the diagonal. Indeed, suppose that the $(i,i)$ entry of $\partial_F$ is $u$, but no other entries in row $i$ or column $i$ are units. Then the $(i,i)$ entry of $\partial_F^2$ is $u^2$ modulo the maximal ideal, which is impossible since $\partial_F^2=0$. Without loss of generality, we assume that the (2,1) entry is a unit.

Let $B$ be the matrix corresponding to the row operations that zero out all other entries in the first column of $\partial_F$.  The matrix $B\partial_FB^{-1}$ has the form
\begin{align*}
B\partial_F B^{-1}
&=
\begin{pmatrix}
0&a_{1,2}&a_{1,3}&\cdots \\
u&a_{2,2}&a_{2,3}&\cdots \\
0&a_{3,2}&a_{3,3}&\cdots\\
0&a_{4,2}&a_{4,3}&\cdots\\
\vdots&\vdots&\vdots&\ddots
\end{pmatrix}.\\
\intertext{
Let $C$ be the matrix corresponding the column operations that zero out all the other entries in the second row of $B\partial_F B^{-1}$.  This is an identity matrix, except in the top row.  It follows that $C^{-1}B\partial_F B^{-1}C$ has the form}
C^{-1}B\partial_F B^{-1}C
&=
\begin{pmatrix}
0&a'_{1,2}&a'_{1,3}&\cdots \\
u&0&0&\cdots \\
0&a'_{3,2}&a'_{3,3}&\cdots\\
0&a'_{4,2}&a'_{4,3}&\cdots\\
\vdots&\vdots&\vdots&\ddots
\end{pmatrix}.\\
\intertext{
The fact that this matrix squares to zero and $u$ is a unit then forces the first row and the second column of this matrix to also be zero, yielding:}
C^{-1}B\partial_F B^{-1}C
&=
\begin{pmatrix}
0&0&0&\cdots \\
u&0&0&\cdots \\
0&0&a'_{3,3}&\cdots\\
0&0&a'_{4,3}&\cdots\\
\vdots&\vdots&\vdots&\ddots
\end{pmatrix}.
\end{align*}
Since $B$ and $C$ are automorphisms of $F$, we conclude that
$
F = F'\oplus T
$
for some differential module $F'$, where $T$ is a rank 2 free $R$-module, and $\del_T = \begin{pmatrix} 0&0\\u&0 \end{pmatrix}.$   Finally, by changing basis on $T$ via the diagonal matrix $(1,u^{-1})$,  we can assume $u = 1$. If $F'$ is minimal, then we are done; otherwise, we iterate.  The process terminates because the rank of $F'$ is strictly smaller than the rank of $F$, which is finite.
\end{proof}

We are now ready to prove our existence and uniqueness results for minimal free resolutions of differential modules:
\begin{thm}
\label{technical}
Let $D \in \DM(R, a)$. 
\begin{itemize}
\item[(a)] If $D$ admits a free flag resolution $F\to D$ such that $F$ is finitely generated, then $D$ admits a minimal free resolution $M\to D$ that is unique up to isomorphism of differential modules, and where $M$ is finitely generated. 
\item[(b)] If $R_0$ is a field, and $D$ admits a free flag resolution $F=\bigoplus F_i$ with $F_i$ finitely generated for all $i$ and such that
\begin{itemize}
\item each graded component $(F)_j$ is finite dimensional over $R_0$, and 
\item the graded components $(F)_j$ vanish for $j \ll 0$, 
\end{itemize}
then $D$ admits a minimal free resolution that is unique up to isomorphism of differential modules.
\end{itemize}
\end{thm}

Note that Theorem \ref{technical} implies Theorem \ref{thm:exist} from the introduction. Indeed, Theorem~\ref{thm:deformation} shows that if $\HH(D)$ has finite projective dimension, then $D$ admits a finite free flag resolution; thus Theorem \ref{technical}(a) implies Theorem \ref{thm:exist}(a).  If, on the other hand, $D$ has degree $0$, then Theorem~\ref{thm:deformation} yields a flag resolution $F$ satisfying the hypotheses of Theorem \ref{technical}(b);  it follows that Theorem \ref{technical}(b) implies Theorem \ref{thm:exist}(b).  

However, the hypotheses of Theorem~\ref{technical} are strictly weaker than those of Theorem~\ref{thm:exist}: Example~\ref{ex:0x00} below shows that $D$ can have a finitely generated flag resolution even when $\HH(D)$ has infinite projective dimension, and any $D\in \DM(R,a)$ will satisfy the hypotheses of Theorem~\ref{technical}(b) if $R=k[x_1,\dots,x_n]/(f)$ is a graded hypersurface ring where $\deg(f)>2a$. 

\begin{proof}[Proof of Theorem \ref{technical}]
We first address the existence statements in parts (a) and (b). Existence in part (a) is immediate from Proposition \ref{prop:minimization procedure}. To prove existence in (b), we use Zorn's Lemma. Let $F$ be a free flag resolution satisfying the conditions in part (b), and let $\Sigma$ be the set of all contractible differential submodules $C$ of $F$ such that the inclusion $C \into F$ of differential modules admits a splitting. The set $\Sigma$ is a nonempty poset under inclusion. Let $C_1 \subseteq C_2 \subseteq \cdots$ be a chain in $\Sigma$. Each $C_i$ is a summand of $C_{i+1}$; for each $i \ge 1$, choose $T_i$ such that $C_{i} = T_{i} \oplus C_{i-1}$ (take $T_1 = C_1$). We claim that $\bigoplus_{i \ge 1} T_i \in \Sigma$. Indeed, we have morphisms $F \to T_i$ for each $i$, and therefore a morphism $F \to \prod_{i \ge 1} T_i$. Since the graded components of $F$ are finite dimensional over $R_0$, this morphism must factor through $\bigoplus_{i \ge 1} T_i$; moreover, the induced map $F \to\bigoplus_{i \ge 1} T_i$ is surjective. It is easy to check that free contractible objects in $\DM(R, a)$ are projective, and therefore this surjection splits. Thus, $\bigoplus_{i \ge 1} T_i \in \Sigma$, and it is clearly an upper bound for our chain. Applying Zorn's Lemma, we choose a maximal element $C$ of $\Sigma$. Write $F = M \oplus C$. If $M$ is not minimal, we can apply Proposition \ref{prop:minimization procedure} to split a contractible summand from $M$, contradicting the maximality of $C$. Thus, $M$ is a minimal free resolution of $D$. 

As for uniqueness, choose two minimal free resolutions $M$ and $M'$ of $D$.  Applying Lemma~\ref{prop:lifting}  to the identity map on $D$, we may choose morphisms of differential modules
$\a \colon M\to M'$ and $\b \colon M'\to M$ and homotopies $h$ and $h'$ such that
\begin{align*}
\b\a - \id_M&= h \del_M + \del_M h \\
\a\b - \id_{M'}&= h' \del_{M'} + \del_{M'}h'.
\end{align*}
Since $\del_M$ and $\del_{M'}$ are minimal, we conclude that $\b \a = \id_M$ and $\a \b = \id_{M'}$ modulo $\m$. 

Now, for uniqueness in part (a): the proof of existence in part (a) shows that we can assume that $M$ is a finite rank free module.  Since $\a \b = \id_{M'}$ modulo $\m$, this implies that $M'$ must also have finite rank.  The fact that $\b$ and $\a$ are isomorphisms follows from Nakayama's Lemma. For (b): the proof of existence in (b) shows that we can assume that the underlying module of $M$ has the form $\bigoplus_{j \in \Z} R(-j)^{\beta_j}$, where $\beta_j=0$ for $j\ll 0$ and is finite for all $j$. Since $\a \b = \id_{M'}$ modulo $\m$, the underlying module of $M'$ is a summand of the underlying module of $M$, and so it has the same form. The fact that $\b$ and $\a$ are isomorphisms now follows from the graded version of Nakayama's Lemma.
\end{proof}

\begin{remark}
\label{difficulty}
We include a few words about the difficulty of loosening the hypotheses on $D$ in Theorem \ref{technical}. Suppose $D$ admits a free flag resolution $F$ such that $F_n$ is finitely generated for all $n$; this is the case, for instance, when the homology of $D$ is finitely generated. In this case, one can use Proposition \ref{prop:minimization procedure} to minimize any finite piece $F_0 \oplus \cdots \oplus F_n$ of the flag $F$. Taking an appropriate limit, one can construct a minimal differential submodule $M$ of $F$ such that the inclusion $M \into F$ is a quasi-isomorphism. But, without the additional finiteness assumptions on $F$ in Theorem \ref{technical}, determining whether $M$ is a summand of $F$ becomes difficult.
%leads to subtle questions concerning endomorphisms of infinite rank free modules.

Extending our uniqueness results creates similar problems. The lifting property in Lemma \ref{prop:lifting} does not immediately imply uniqueness unless one can apply Nakayama's Lemma to $M$. For instance, if $R$ is a (trivially graded) local ring, and $M$ is a free $R$-module with countably infinite basis, there exist maps $M\to M$ that are the identity modulo $\mathfrak m$ but that fail to be surjective. This problem doesn't arise when proving the uniqueness of minimal free resolutions of modules, because 
such resolutions are finitely generated in each homological degree, and chain maps respect homological degree. %But differential modules have no homological grading, and so they lack this rigidity.
\end{remark}

%%%%%%%%%%%%%%%%%%%%%
\section{Examples and Applications}\label{sec:examples and applications}
%%%%%%%%%%%%%%%%%%%%%

%%%%%%%%%%%%%%%%%%%%%%%%
\subsection{Betti numbers of differential modules}
\label{betti}
%%%%%%%%%%%%%%%%%%%%%%%%%
Given $D \in \DM(R, a)$ and an $R$-module $N$, we define
$$
\Tor^R_{\DM}(D, N) = H(F \otimes_R N), 
$$
where $F$ is a free resolution of $D$, and $F \otimes_R N$ is equipped with the differential $\del_F \otimes \id$. It follows from Lemma \ref{prop:lifting} that this definition doesn't depend on the choice of $F$. We define the \defi{Betti numbers} of $D$ to be given by
$$
\beta_j^{\DM}(D) = \dim_k \Tor^R_{\DM}(D, k)_j,
$$
where $\Tor^R_{\DM}(D, k)_j$ is the degree $j$ part of the $\Tor$ group.  Recall from Remarks \ref{remarks} that, if $D$ admits a minimal free resolution $M$, then $\beta_j^{\DM}(D)$ is the number of degree $j$ elements in any basis of $M$.

Given an ordinary $R$-module $N$, denote its Betti numbers by $\beta_{i, j}(N):= \dim_k \Tor_i(N,k)_j$. We have the following result bounding the Betti numbers of $D$ in terms of those of $H(D)$.

\begin{cor}\label{cor:semicontinuity}
Let $D \in \DM(R, a)$, and suppose $\HH(D)$ is finitely generated.  We have
\[
\beta_{j}^{\DM}(D) \leq \sum_{i=0}^\infty \beta_{i,j+ia}(\HH(D)).
\]
If $a=0$, then we also have
\[
 \beta_{j}^{\DM}(D) \equiv \sum_{i=0}^\infty \beta_{i,j}(\HH(D)) \ \ (\operatorname{mod} 2).
 \]
\end{cor}

\begin{proof}
This follows in a straightforward manner from Theorem~\ref{thm:deformation a}. In more detail, let $F$ be a free flag resolution of $D$ as in Theorem~\ref{thm:deformation a}. Write $b_j:=\sum_{i=0}^\infty \beta_{i,j+ia}(\HH(D))$, so that $F=\bigoplus_{j} R(-j)^{b_j}$. By definition, $\beta_j^{\DM}(D)$ is the number of copies of $k(-j)$ in the homology of
\[
F(-a)/\mathfrak m \overset{\overline{\del}}{\to} F/\mathfrak m\overset{\overline{\del}}{\to} F(a)/\mathfrak m.
\]
Focusing on the degree $j$ part, we see that $\beta_j^{\DM}(D)$ equals the rank of the homology of
\[
k(-j+a)^{b_{j-a}}(-a) \overset{\alpha}{\to} k(-j)^{b_j}  \overset{\alpha'}{\to} k(-j-a)^{b_{j+a}}(a),
\]
where $\alpha,\alpha'$ are the restrictions of $\overline{\del}$ to the relevant summands.  Thus $\beta^{\DM}_j(D) = b_j - \rank \alpha - \rank \alpha'$.  Both statements follow.
\end{proof}

\begin{example}
Let $D$ be a degree $0$ differential module such that $\HH(D)$ has a pure resolution; that is, the minimal free resolution of $\HH(D)$ has the form $F_0\gets F_1\gets \cdots$, where each $F_i$ is generated in a different degree.  Then the differential $\del_F$ in Theorem~\ref{thm:deformation} must be minimal, and thus the Betti numbers of $D$ are the same as those of $\HH(D)$.  This begs the question: is there a Boij-S\"oderberg theory~\cite{boij-sod,es-jams,floystad} for differential modules?
\end{example}

\begin{example}
Let $D\in \DM(R,0)$.  If, for every $j\in \ZZ$, there is at most one integer $i$ such that $\beta_{i,j}(\HH(D))\ne 0$, then the Betti numbers of $D$ equal those of $\HH(D)$.  For instance, let $R=k[x_1,\dots,x_n]$, and assume that $\HH(D)=R/(f_1,\dots,f_c)$, where $f_1, \dots, f_c$ is a regular sequence such that $\deg(f_i) =2^i$ for all $1\leq i\leq c$.  We have that $\beta_j^{\DM}(D) = 1$ for all $0\leq j < 2^{c+1}$.  
\end{example}

\begin{example}
The Betti numbers of differential modules can behave quite differently as one varies the degree $a$. For instance, let $R=k[x]/(x^2)$, where $\deg(x) = 1$. For any $a\in \Z$,  we can consider the differential module $D$ given by $k\overset{0}{\longrightarrow} k(a)$ as an object in $\DM(R, a)$. This has the following Betti numbers, as $a$ varies from 0 to 2:
$$
\begin{tabular}{ l | c  }
  $a$ & $\beta_j(D)$ \\ \hline
  0 & $1 \text{ for all } j \geq 0 \text{, and } 0 \text{ for all } j<0$ \\
  1 & $\infty \text{ for $j$ = 0, and } 0 \text{ for all } j\ne 0$ \\
  2 &  $1 \text{ for all } j \leq 0, \text{ and } 0 \text{ for all } j>0$.  \\
\end{tabular}
$$
\iffalse
\begin{itemize}
\item $\text{If $a=0$, then } \beta_j(D) = 1 \text{ for all } j \geq 0 \text{ and } \beta_j(D) = 0 \text{ for all } j<0.$
\item $\text{If $a=1$, then }  \beta_0(D) = \infty \text{ and } \beta_j(D) = 0 \text{ for all } j\ne 0.$
\item $\text{If $a=2$, then }  \beta_j(D) = 1 \text{ for all } j \leq 0 \text{ and } \beta_j(D) = 0 \text{ for all } j>0.$
\end{itemize}
\fi
\end{example}
\iffalse
\begin{example}
Let $R=k[x,y]/(x^2,y^6)$.  Let $D\in \DM(R,2)$ be $R/(x)\oplus R/(y^3)$ with the zero differential.  Then the Betti numbers of $D$ satisfy $\beta_j(D)=1$ for all $j\in \ZZ$.
\end{example}
\fi

\begin{remark}
Assume that $R_0$ is a field and that $R$ is generated in degree 1.
We set $\reg^{\DM}(D):= \max\{ j \text{ : } \beta^{\DM}_j(D) \ne 0\}$.  We already noted in Remark~\ref{rmk:a matters} that the categories $\DM(R,a)$ can differ as one varies $a$.  Consider the following observations:
\begin{itemize}
	\item[(1)]  Every finitely generated $D\in \DM(R,0)$ satisfies $\reg^{\DM}(D)<\infty \iff R$ is regular.
	\item[(2)]  Every finitely generated $D\in \DM(R,1)$ satisfies $\reg^{\DM}(D)<\infty \iff R$ is Koszul.
\end{itemize}
The first statement is immediate from Theorem \ref{thm:deformation a}, and the second statement follows from Theorem \ref{thm:deformation a} and \cite[Theorem 1 and Corollary 3]{AP}. In particular, the properties of $\DM(R,a)$ for different $a$ seem to be related to interesting homological phenomena.
\end{remark}

%%%%%%%%%%%%%%%%%%%%%
\subsection{Minimal free flag resolutions}
%%%%%%%%%%%%%%%%%%%%%
When $D\in \DM(R,0)$, $D$ has finitely generated homology, and $R_0$ is a field, one need not choose between a minimal resolution and a flag resolution.  

\begin{prop}\label{prop:minimal and flag}
Assume $R_0$ is a field and that $D$ is an object in $\DM(R,0)$ with finitely generated homology.  Let $d$ be any integer such that the minimal free resolution $M \xra{\simeq} D$ is generated in degree $\geq d$. The resolution $M$ admits a flag structure given by the splitting $M=\bigoplus_{j\in \ZZ} M_j$, where $M_j=R(-j-d)^{\beta_{j+d}^{\DM}(D)}$.
\end{prop}
\begin{proof}
After replacing $D$ by $D(d)$, we can assume that $d=0$ and that $M$ is generated in degree $\geq 0$.  With this hypothesis, $M_j$ is generated entirely in degree $j$.  Since the differential on $M$ is minimal, it must send $M_j$ to $\bigoplus_{i<j} M_i$, thus yielding a flag structure.
\end{proof}
\begin{remark}\label{rmk:compute}
In the situation of Proposition~\ref{prop:minimal and flag}, the following variant of Construction~\ref{killing cycles} can be used to iteratively produce the minimal free resolution. Since $H(D)$ is finitely generated, we can twist by some $d\in \ZZ$ so that $\HH(D)$ lives in entirely in  degrees $\geq 0$.  Let $b_0:=\rank \HH(D)_0$, and let $\epsilon_0\colon R^{b_0}\to D$ be a map whose image descends to a basis of $\HH(D)_0$.  Then $\cone(\epsilon_0)$ is a finitely generated differential module whose homology lives entirely in degrees $\geq 1$.  We then define $\epsilon_1\colon R(-1)^{b_1}\to \cone(\epsilon_0)$ analogously.  The resulting free differential module is a flag by construction and minimal for degree reasons.
\end{remark}

\begin{example}
\label{DNE}
We return to Example \ref{notflag}, where $R = \QQ[x, y]$, $D = R^{\oplus 2}$, and $\del_D \colon R^{\oplus 2} \to R(2)^{\oplus 2}$ is the matrix
$$
\begin{pmatrix}
xy & -x^2 \\
y^2 & -xy\\
\end{pmatrix}.
$$
We saw in Example \ref{notflag} that $D$ does not admit the structure of a free flag. We will now show that $D$ does not even admit a free flag resolution that is minimal.  By uniqueness of minimal free resolutions, it suffices to check that $D$ is its own minimal free resolution.  There is a free flag resolution $\epsilon\colon F \xra{\simeq} D$, where $F = R^{\oplus 4}$,
$$
\del_F = \begin{pmatrix} 0 & -y & -x & -1 \\ 0 & 0 & 0 & -x \\ 0 & 0 & 0 & y \\ 0 & 0 & 0 & 0 \end{pmatrix},
\qquad
\text{ 
and
}
\qquad
\epsilon = \begin{pmatrix} x &  -1 & 0 & 0 \\ y & 0 & 1 & 0 \end{pmatrix}.
$$

\[
\text{If } \qquad
A = \begin{pmatrix} -1 & 0 & 0 & 0 \\ x & -1 & 0 & 0 \\ y & 0 & 1 & 0 \\ 0 & -y & -x & -1\end{pmatrix},
\qquad
\text{ then }
\qquad
A \del_F A^{-1} = \begin{pmatrix} 0 & 0 & 0 & -1 \\ 0 & xy&-x^2&0 \\ 0 & y^2 & -xy&0 \\ 0 & 0 & 0 & 0 \end{pmatrix}.
\]
In particular, $F$ is isomorphic to the direct sum of $D$ and a contractible differential module.  Since $\del_D$ is minimal, this implies that $D$ is its own minimal free resolution.
\end{example}

%%%%%%%%%%%%%%
\subsection{The structure of special minimal free resolutions}
%%%%%%%%%%%%%%%

Theorem~\ref{thm:deformation} allows one to use results about minimal free resolutions of modules to prove structural results for certain special families of differential modules.  Corollary~\ref{cor:HB} provides one such example, and we now provide a brief proof of that fact.
\begin{proof}[Proof of Corollary~\ref{cor:HB}]
Theorem~\ref{thm:deformation} combined with the Hilbert-Burch Theorem~\cite[Theorem~20.15]{eisenbud} provides a flag free resolution of $D$ of the form
\begin{equation}\label{eqn:HB}
\bordermatrix{
&R&F_1&F_2\cr
R&0&\del_{1,0}&\del_{2,0}\cr
F_1&0&0&\del_{2,1}\cr
F_2&0&0&0
}.
\end{equation}
For degree reasons, the free module $F_2$ cannot have any generators of degree $0$, and thus $\del_{2,0}$ must be minimal.  Since $\del_{1,0},\del_{2,1}$ came from a minimal free resolution, this implies that the entire differential is minimal, and so $F$ is the minimal free resolution of $D$.
\end{proof}
\begin{example}
In the setup of Corollary~\ref{cor:HB}, there can be many non-isomorphic differential modules $D$ with a fixed $\HH(D)$.  For instance, let $R=k[x_1,\dots,x_n]$ be a standard graded polynomial ring over a field, and assume that $\HH(D) = R/(x_1,x_2)$.  By Theorem~\ref{thm:deformation}, every such $D\in \DM(R,0)$ has a minimal free resolution of the form
\[
\begin{pmatrix} 0 & x_1 & x_2 & q \\ 0 & 0 & 0 & -x_2 \\ 0 & 0 & 0 & x_1 \\ 0 & 0 & 0 & 0 \end{pmatrix},
\]
where $q\in R$ is some quadratic form.  By changing basis, we can eliminate any terms in $q$ involving $x_1$ or $x_2$, and we see that such modules $D$ are in bijection with quadratics in $x_3, \dots, x_n$.
\end{example}

\begin{example}
Now, suppose $D$ is as in Corollary \ref{cor:HB}, except we allow the degree of the differential on $D$ to possibly be nonzero. Again, by combining Theorem~\ref{thm:deformation} with the Hilbert-Burch Theorem~\cite[Theorem~20.15]{eisenbud}, we obtain a free flag resolution of the form \eqref{eqn:HB}.  If $F_2$ has a degree $0$ generator, which is only possible when $a>1$, then this resolution may fail to be minimal. This is precisely what happens in Example~\ref{DNE}.
\end{example}

We obtain a similar structural result when $\HH(D)$ is a Gorenstein codimension $3$ algebra.
\begin{example}\label{ex:gor codim 3}
Let $R$ be a polynomial ring over a field, and let $D\in \DM(R,0)$ such that $\HH(D)$ is a Gorenstein codimension $3$ quotient of $R$. Theorem~\ref{thm:deformation}, combined with the Buchsbaum--Eisenbud structure theorem for such algebras~\cite{buchs-eis-gor}, implies that $D$ has a flag resolution of the form
\[
\bordermatrix{
&R&F_1&F_2&R(-e)\cr
R&0&\del_{1,0}&\del_{2,0}&\del_{3,0}\cr
F_1&0&0&\del_{2,1}&\del_{3,1}\cr
F_2&0&0&0&\del_{3,2}\cr
R(-e)&0&0&0&0
},
\]
where $\del_{2,1}$ is a $(2n+1)\times (2n+1)$ skew-symmetric matrix, $\del_{1,0}$ consists of the $2n\times 2n$ principal Pfaffians of $\del_{2,1}$, and $\del_{3,2}$ is the transpose of $\del_{1,0}$.  For degree reasons, neither $F_2$ nor $F_3$ have summands of degree $0$ or $e$, and so $F$ must equal the minimal free resolution of $D$.  Many non-trivial examples of this form exist.  For instance, one can freely choose $\del_{3,0}$ and $\del_{2,0}$ and then set $\del_{3,1}$ to be the negative transpose of $\del_{2,0}$.
\end{example}

\begin{example}\label{ex:0x00}
Theorem~\ref{thm:deformation} shows that if $\HH(D)$ has finite projective dimension, then the minimal free resolution of $D$ is finitely generated.  It can also be the case that the minimal free resolution of $D$ is finite even when $\HH(D)$ has infinite projective dimension.  For example, if $R=k[x]/(x^2)$ with $\on{deg}(x) = 1$, consider $D=R(-a)\oplus R(-1)\in \DM(R,a)$ with the differential
\[
\begin{pmatrix}
0&x\\
0&0
\end{pmatrix}.
\]
The homology of $D$ is $k(-a)\oplus k(-2)$, which has infinite projective dimension.  By contrast, $D$ is the minimal free resolution of itself.
\end{example}

\bibliographystyle{amsalpha}
\bibliography{Bibliography}

\end{document}